\documentclass{article}

\usepackage{arxiv}

\usepackage{xcolor}
\usepackage{todonotes}
\usepackage{amsmath,tikz-cd}

\usepackage{amsmath,amsfonts,mathtools}
\usepackage{amssymb}
\usepackage{enumerate}
\usepackage{fancyhdr}
\usepackage{enumitem}
\usepackage{verbatim}
\usepackage{anysize}
\usepackage{multicol}
\usepackage{amsmath,amsthm}
\usepackage{amsmath}
\usepackage{mathrsfs}

\usepackage[utf8]{inputenc} 
\usepackage[T1]{fontenc}    
\usepackage{hyperref}       
\usepackage{url}            
\usepackage{booktabs}       
\usepackage{amsfonts}       
\usepackage{nicefrac}       
\usepackage{microtype}      

\usepackage{graphicx}
\usepackage{biblatex}
\usepackage{doi}

\newtheorem{theorem}{\bf Theorem}[section]
\newtheorem{corollary}[theorem]{\bf Corollary}
\newtheorem{definition}[theorem]{\bf Definition}
\newtheorem{example}[theorem]{\bf Example}
\newtheorem{lemma}[theorem]{\bf Lemma}
\newtheorem{proposition}[theorem]{\bf Proposition}

\title{Abstract Hardy inequalities: The case $p=1$}

\author{ \href{https://orcid.org/0009-0000-8958-998X}{\includegraphics[scale=0.06]{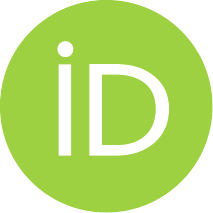}\hspace{1mm}Alejadro Santacruz Hidalgo} \\
	Department of Mathematics\\
	University of Western Ontario\\
	London, ON, Canada \\
	\texttt{asantacr@uwo.ca} \\
}

\renewcommand{\shorttitle}{\textit{arXiv} Template}

\DeclareMathOperator{\essinf}{ess \ inf}

\hypersetup{
pdftitle={A template for the arxiv style},
pdfsubject={q-bio.NC, q-bio.QM},
pdfauthor={David S.~Hippocampus, Elias D.~Striatum},
pdfkeywords={First keyword, Second keyword, More},
}

\begin{document}
\maketitle

\begin{abstract}
        The Boundedness of an abstract formulation of Hardy operators between Lebesgue spaces over general measure spaces is studied and, when the domain is $L^1$, shown to be equivalent to the existence of a Hardy inequality on the half line with general Borel measures. This is done by extending the greatest decreasing minorant construction to general measure spaces depending on a totally ordered collection of measurable sets, called an ordered core. A functional description of the greatest decreasing minorant is given, and for a large class of ordered cores, a pointwise description is provided. As an application, characterizations of Hardy inequalities for metric measure spaces are given, we note that the metric measure space is not required to admit a polar decomposition. 
\end{abstract}

\maketitle



\section{Introduction: Abstract Hardy inequalities}

Given three Borel measures on $[0,\infty)$, simple necessary and sufficient conditions for which the inequality
\begin{equation}\label{hardymeasures}
\Bigg( \int\limits_{[0,\infty)} \bigg( \int\limits_{[0,x]} f \, d\lambda \bigg)^{q} \, d \nu(x)  \Bigg)^{1/q} \leq C \Big( \int\limits_{[0,\infty)} f^p \, d\eta \Big)^{1/p}    
\end{equation}
holds for all positive measurable functions have been given by several authors. Letting $p = q> 1$, $\lambda$ and $\eta$ as the Lebesgue measure and $d\nu = 1/x \, d\lambda$ yields the classical Hardy inequality proved in the 1925 paper \cite{hardy1925}, which holds with best constant $p/(p-1)$. Muckenhoupt, in \cite{muckenhoupt72}, showed that letting $\nu$ and $\eta$ be absolutely continuous with respect to the Lebesgue measure, the inequality holds if and only if a one-parameter supremum is finite. Bradley, in \cite{bradley78}, extended the result for indices $1 < p \leq q < \infty$. Maz'ya, in \cite{mazya85} and Sinnamon, in \cite{sinnamon91}, showed that for $0 < q < p$ and $1 < p < \infty$, the characterization is given by the finiteness of a single integral. In the case $p > 1$, simple characterizations for inequality (\ref{hardymeasures}) can be found in \cite{sinnamon04}. 

Extensions have been made in several directions; results for more general measures, higher dimensions, and restrictions on the domain are available, see \cite{maligranda07}. 

The case $p = 1$ must be treated differently. In \cite[Theorem~3.1]{stepanov08} the following characterization is shown:

\begin{theorem}\label{stepanov}
    If $0 < q < 1 = p$, then the inequality (\ref{hardymeasures}) holds if and only if 
\begin{equation}\label{hardymeasures2}
    \Bigg( \int\limits_{[0,\infty)} \bigg( \int\limits_{[0,x]} \frac{1}{\underline{w}} \, d\nu \bigg)^{\frac{q}{1-q}} \, d \nu(x)  \Bigg)^{1/q} < \infty,
\end{equation}
with $\underline{w}(x) = \essinf_\lambda\{ w(t) : t \in [0,x] \}$, where $d\eta = d\lambda^\perp + w d\lambda$ and $\lambda^\perp \perp \lambda$.
\end{theorem}

In this paper we are concerned with a large class of Hardy inequalities introduced in \cite{sinnamon22}, which require the following definition. 
\begin{definition}
Let $(U,\Sigma,\mu)$ and $(Y,\mathcal{T},\tau)$ be two $\sigma$-finite measure spaces, a map $B: Y \to \Sigma$ is called a \textit{core map} provided it satisfies:
\begin{enumerate}
    \item (Total order) The range of $B$ is totally ordered by inclusion.
    \item (Measurability) For each $E \in \Sigma$ the map $y \mapsto \mu(E \cap B(y))$ is $\mathcal{T}$-measurable.
    \item ($\sigma$-boundedness) There is a countable subset $Y_0 \subseteq Y$ such that $\bigcup_{y \in Y} B(y) = \bigcup_{y \in Y_0} B(y)$.
    \item (Finite measure) For all $y \in Y$, $\mu(B(y)) < \infty$.
\end{enumerate}
\end{definition}

Given a core map, an inequality of the form
\begin{equation}\label{abstracthardy1}
    \Bigg( \int\limits_{Y} \bigg( \int\limits_{B(y)} f \, d\mu \bigg)^{q} \, d \tau(y)  \Bigg)^{1/q} \leq C \Big( \int\limits_{U} f^p \, d\eta \Big)^{1/p},    
\end{equation}
for all positive measurable functions $f$ is called an \textit{Abstract Hardy inequality}. Notice that setting $Y = U = [0,\infty)$ and $B(y) = [0,y]$ recovers inequality (\ref{hardymeasures}). In the case that $\mu = \eta$,  \cite[Theorem~2.4]{sinnamon22} shows that the best constant $C$ in (\ref{abstracthardy1}) is the same as the best constant in the inequality \begin{equation*}
\Bigg( \int_0^\infty \bigg( \int_{0}^{b(x)} f(t) \, dt  \bigg)^{q} \, dx  \Bigg)^{1/q} \leq C \Big( \int\limits_0^\infty f(t)^p \, dt \Big)^{1/p}, \quad\mbox{for all }f\in L^+, 
\end{equation*}
for an appropriate non-increasing function $b: (0,\infty) \to [0,\infty]$. For $p > 1$, any abstract Hardy inequality (\ref{abstracthardy1}) can be reduced to the case where $\eta$ and $\mu$ coincide (see \cite[Theorem~5.1]{sinnamon22}), however, the reduction is not available for the case $p=1$, as the formula involves a power of the form $\frac{1}{p-1}$. Our main result is the following extension of Theorem \ref{stepanov} to the abstract setting. \\

\

\textbf{Theorem A}\label{mainresult}
        For $\sigma$-finite measure spaces $(Y,\mathcal{T},\tau), (U,\Sigma,\mu), (U,\Sigma,\eta)$ and a core map $B:Y \to \Sigma$, let $\eta = \eta_a + \eta_s$, where $d\eta_a = u d \mu$ and $\eta_s \perp \mu$. Then the best constant $C$ in the inequality
\begin{equation}\label{Tabstracthardy1}
    \Bigg( \int\limits_{Y} \bigg( \int\limits_{B(y)} f \, d\mu \bigg)^{q} \, d \tau(y)  \Bigg)^{1/q} \leq C \int\limits_{U} f \, d\eta ,    
\end{equation}
satisfies
$$
C \approx \Bigg( \int\limits_{Y} \bigg( \int\limits_{\mu(B(z)) \leq \mu(B(y))} R \left( \frac{1}{\underline{u}} \right) \circ \mu \circ B(y) \, d\tau(y)   \bigg)^{\frac{q}{1-q}} \, d\tau(z) \Bigg)^{\frac{1-q}{q}}, 
$$
for $q \in (0,1)$ and
$$
C = \sup_{s \in U} \left( \frac{1}{\underline{u}}(s) \right) \tau\left( \left\{ y \in Y: s \in B(y)\right\} \right)^{1/q}, \text{ for } q \in [1,\infty).  
$$ 

Here the least core decreasing majorant $\underline{u}$ is taken with respect to the core $\mathcal{A} = \{\emptyset\} \cup \{ B(y): y \in Y\}$ $R$ is the transition map from Definition \ref{transitionmaps}.

\, 

\, 

Our approach is to show that, for $p=1$, an abstract Hardy inequality is equivalent to a Hardy inequality with measures and give necessary and sufficient conditions for such an inequality to hold.

In Section 2 we introduce the tools necessary to state our main result. The key construction is the greatest core decreasing minorant of a function, which extends the construction $\underline{w}$ of Theorem \ref{stepanov} to general measure spaces. This construction allows us to reduce inequality (\ref{abstracthardy1}) to a suitable inequality of the form (\ref{hardymeasures}). This is done in Section 3. In Section 4 we give explicit examples of the greatest core decreasing minorant and apply the main result in Section 3 for Hardy inequalities in metric measure spaces. We leave Section 5 for the proof of a functional description of the least core decreasing minorant, which is the key step in proving our main result.

We finish this introduction by setting up notation and some basic results. For a $\sigma$-finite measure space $(U,\Sigma,\mu)$ and a set $\mathcal{A} \subseteq \Sigma$ we denote the $\sigma$-ring generated by $\mathcal{A}$ by $\sigma(\mathcal{A})$. By $L(\mathcal{A})$ we mean the collection of all (equivalence classes of) $[-\infty,\infty]$-valued $\sigma(\mathcal{A})$-measurable functions on $U$. The collection of non-negative functions in $L(\mathcal{A})$ is written as $L^+(\mathcal{A})$. We reserve the notation $L^{0}_\mu$ for the collection of $\Sigma$-measurable functions and $L^{+}_\mu$ for the non-negative ones.

We write $0 \leq \alpha_n \uparrow \alpha$ to indicate the limit of a non-decreasing sequence in $[0,\infty]$ and use $\alpha_n \downarrow \alpha$ when the sequence is non-increasing. In the case of sets, we write $A_n \uparrow A$ or $A_n \downarrow A$ if their characteristic functions converge increasingly or decreasingly almost everywhere. We adopt the convention that expressions that evaluate to $0/0$ will be taken to be zero. For  $p \in (0,\infty]$ the expression $L^{p}_\mu$ denotes the usual Lebesgue space of $\mu$-measurable functions. For two positive constants $C$ and $D$ we write $C \approx D$ if $d_1 D \leq C \leq d_2 D$ for positive numbers $d_1,d_2$.

For a function $f \in L(\Sigma)$, its distribution function, $\mu_f$ is given by $$\mu_f(\alpha) = \mu \left( \left\{ s \in U: \left| f(s)\right| > \alpha \right\} \right).$$ Following \cite{bs}, if $\mu_f = \tau_g$ then for any $p \in (0,\infty)$ we have $\int\limits_{U} \left| f\right|^p \, d\mu = \int\limits_{Y} \left| g\right|^p \, d\tau$.

We consider a \textit{metric measure space} to be the triple $(\mathbb{X},d,\mu)$ where $d$ is a distance function and $\mu$ is a Borel measure with respect to the topology induced by the metric $d$ and for every $a \in \mathbb{X}$ and $r > 0$, the closed ball of radius $r$ centered at $a$ has finite measure.

\section{Ordered cores}
In this section, we set up our tools and notation to work with monotone functions in general measure spaces without an order relation on the elements. First, we recall some key definitions in \cite[Definition~1.1]{coredecreasing}:

\begin{definition}\label{def_core}
    Let $(U,\Sigma,\mu)$ be a $\sigma$-finite measure space. A family of sets $\mathcal{A} \subseteq \Sigma$ is a \textit{full $\sigma$-bounded ordered core} provided: 
\end{definition}
\begin{enumerate}
    \item The family $\mathcal{A}$ is totally ordered by inclusion.
    \item Every set $E \in \mathcal{A}$ has finite $\mu$-measure.
    \item The space $U$ can be realized as the union $U = \bigcup_{E \in A_0} E$ for some countable subfamily $A_0$ of $\mathcal{A}$.
\end{enumerate}

We will also need the following related concepts

\begin{itemize}
    \item For a full ordered core $\mathcal{A}$ the relation $\leq_{\mathcal{A}}$ on $U$ is defined by $u \leq_{\mathcal{A}} v$ if for all $A \in \mathcal{A}$, $v \in A$ implies $u \in A$. When there is no ambiguity on the core, we omit the subscript $\mathcal{A}$. We will write $u <_{\mathcal{A}} v$ whenever $u \leq_{\mathcal{A}} v$ holds but $v \leq_{\mathcal{A}} u$ fails. 
    
    \item For a full ordered core $\mathcal{A}$ there exists an extension $\mathcal{M}$ that does not modify the order relation and is closed under arbitrary unions and intersections, provided the result has finite measure and $\sigma(\mathcal{A}) = \sigma(\mathcal{M})$ (see \cite[Lemma~4.1]{coredecreasing}). We will refer to this extension as the \textit{maximal core} induced by $\mathcal{A}$. 

    \item For a maximal core $\mathcal{M}$ and $E \in \sigma(\mathcal{A})$, then $E \in \mathcal{M}$ is equivalent to: For all $u,v \in U$, if $v \in E$ and $u \leq_\mathcal{A} v$, then $u \in E$. 
    (see \cite[Lemma~4.1 (c)]{coredecreasing})
    
    \item A function $f: U \to [0,\infty]$ is  called \textit{core-decreasing} relative to $\mathcal{A}$ if it is $\sigma(\mathcal{A})$-measurable and if for all $u,v \in U$, $u \leq_\mathcal{A} v$ implies $f(u) \geq f(v)$. The collection of core-decreasing functions is denoted by $L^{\downarrow}(\mathcal{A})$.
\end{itemize}

We define the collection of (equivalence classes of) functions
$$
L^{1}_{\text{loc}_{\mathcal{A}},\mu} = \left\{ f \in L(\Sigma): \int_{A} \left|f \right| \, d\mu < \infty \text{ for all } A \in \mathcal{A} \right\}.
$$

Let $\mathcal{B}$ be the Borel $\sigma$-algebra on $[0,\infty)$. Then by virtue of \cite[Theorem~6.4]{coredecreasing}, for every ordered core $\mathcal{A}$ there exists a Borel measure $\lambda$ induced by the core $\mathcal{A}$ and linear maps $R: L^{1}_{\text{loc}_{\mathcal{A}},\mu} \to L^{1}_{\text{loc},\lambda}$ and $Q: L^{1}_{\text{loc},\lambda} \to  L^{1}_{\text{loc}_{\mathcal{A}},\mu}$ satisfying:
\begin{enumerate}
    \item If $\varphi \in L^{+}(\mathcal{B}) \cup L^{1}_{\text{loc},\lambda}$, then $RQ \varphi = \varphi$ up to a set of $\lambda$-measure zero.
    \item If $f \in L^{+}(\mathcal{A}) \cup \left( L^{1}_{\text{loc}_\mathcal{A},\mu} \cap L(\mathcal{A}) \right)$, then $QR f = f$ up to a set of $\mu$-measure zero.
    \item If $f \in L^{+}(\Sigma)$, $\varphi \in L^{+}(\mathcal{B})$ and $A \in \mathcal{A}$ then
    $$
    \int\limits_{A} f Q(\varphi) \, d\mu = \int\limits_{[0,\mu(A)]} R(f) \varphi \, d\lambda \quad\mbox{and}\quad \int\limits_{U} f Q(\varphi) \, d\mu = \int\limits_{[0,\infty)} R(f) \varphi \, d\lambda.  
    $$
    \item If $f,g \in L^{+}_\mu \cap L^{1}_{\text{loc}_\mathcal{A},\mu} \cap L(\mathcal{A})$, then $R(fg) = R(f)R(g)$.
    \item If $f,g \in L^{+}_\mu \cap L^{1}_{\text{loc}_\mathcal{A},\mu} \cap L(\mathcal{A})$  satisfy $\int_A f \, d\mu = \int_{A} g \, d\mu$ for all $A \in \mathcal{A}$, then $f = g$ up to a set of zero $\mu$-measure.
\end{enumerate}

Notice that condition (v) follows from (ii) and the fact that the equality $$\int\limits_{[0,x]} Rf \, d\lambda = \int\limits_{[0,x]} Rg \, d\lambda$$ holding for all $x > 0$ forces that the functions $Rf$ and $Rg$ to be equal $\lambda$-almost everywhere. We reserve a special name for the operators $R,Q$.

\begin{definition}\label{transitionmaps}
    For a $\sigma$-finite measure space $(U,\Sigma,\mu)$ with a $\sigma$-bounded full ordered core $\mathcal{A}$, we denote \textit{transition maps} the operators $R$ and $Q$ mentioned above.  
\end{definition}

We introduce our main technical tool, which extends the greatest non-increasing minorant (see \cite[Section~2]{sinnamon03}).

\begin{definition}\label{cdm}
    For a $\Sigma$-measurable function $g$, we call $h \in L^{\downarrow}(\mathcal{A})$ a greatest core decreasing minorant of $g$ if $0 \leq h \leq \left|g \right|$ $\mu$-a.e and for any $w \in L^{\downarrow}(\mathcal{A})$ satisfying $0 \leq w \leq \left| g \right|$, then $w \leq h$ $\mu$-a.e.
\end{definition}
Note that a greatest core decreasing minorant is unique almost everywhere, provided it exists. The next lemma shows that such a greatest core decreasing minorant always exists.

\begin{lemma}\label{existence}
    Every $\Sigma$-measurable function $g$ admits a greatest core-decreasing minorant denoted $\underline{g}$, which is unique up to a set of $\mu$ measure zero.
\end{lemma}
\begin{proof}    
Suppose that $\left| g \right| \leq C < \infty$ and let $\{ A_n \}_{n \in \mathbb{N}} \subseteq \mathcal{A}$ such that $A_n \uparrow U$. Set $$\alpha_n = \sup\left\{ \int_{A_n} h \, d\mu:  h \in L^{\downarrow}(\mathcal{A}) \text{ and } h \leq \left|g\right|  \right\}.$$ The collection defining the supremum is not empty as $h = 0$ is a core-decreasing function, moreover, the supremum is finite since $\int\limits_{A_n} h \, d\mu \leq C \mu(A_n) < \infty$. 

Let $h_n = 0$ if $\alpha_n = 0$, otherwise there exists $h_n \in L^{\downarrow}(\mathcal{A})$ such that $h_n \leq \left| g \right|$ and $\alpha_n - 1/n < \int\limits_{A_n} h_n \, d\mu$. Since the pointwise maximum of core decreasing functions is core decreasing, we may assume that $\{h_n\}$ is an increasing sequence. Let $h = \sup_{n} h_n$, which is clearly a core decreasing minorant of $g$. 

To show that $h$ is the greatest core decreasing minorant of $g$, let $w$ be another core decreasing minorant, then so is $\max\{h,w\}$, thus $$ \infty > \int\limits_{A_n} h \, d\mu \geq \int\limits_{A_n} h_n \, d\mu > \alpha_n - 1/n \geq \int\limits_{A_n} \max\{w,h\} \, d\mu - 1/n.$$ Then $1/n \geq \int\limits_{A_n} \left(\max\{w,h\}-h \right) \, d\mu \geq 0$. Let $n \to \infty$ to get $\max\{w,h\} = h$ almost everywhere. This completes the proof in the case that $g$ is bounded.

For the unbounded case, define $g_m = \min\{m,\left|g \right|\}$ and let $\underline{g_m}$ be its greatest core decreasing minorant which exists since $g_m$ is bounded. Since $\underline{g_{m-1}} \leq \min\{m-1,\left| g \right|\} \leq \min\{m,\left| g \right|\} = g_m$, then $\underline{g_{m-1}} \leq \underline{g_m}$. Therefore $\{ \underline{g_m}\}_{m \in \mathbb{N}}$ is an increasing sequence. 

Let $h = \sup_{m \in \mathbb{N}} \underline{g_m}$. Since each $\underline{g_m}$ is bounded above by $\left|g\right|$, then $h \leq \left| g\right|$, thus $h$ is a core decreasing minorant of $\left|g\right|$. If $w$ is another core decreasing minorant of $\left|g\right|$, then $\min(m,w)$ is a core decreasing minorant of $\left|g_m\right|$, thus $\min(m,w) \leq \underline{g_m}$. Let $m \to \infty$ to get $w \leq h$ and complete the proof.

\end{proof}

The next theorem gives a functional description of the greatest core decreasing minorant; it extends the corresponding statement in \cite[Theorem~2.1]{sinnamon03} to a very large class of functions. The proof follows a different argument than its real line counterpart and is left for Section 5.

\begin{theorem}\label{pushmass}
    For $\Sigma$-measurable non-negative functions $f$ and $u$, then
$$
\int_U f \underline{u} \, d\mu = \inf\left\{ \int_U g u \, d\mu : \int_E g \, d\mu \geq \int_E f \, d\mu \text{ for all } E \in \mathcal{A}\right\}
$$  
\end{theorem}

As the necessary and sufficient conditions for the existence of a finite constant $C$ in the abstract Hardy inequality (\ref{abstracthardy1}) depend on the computation of this greatest core decreasing minorant, the next result gives an explicit pointwise formula of this minorant, in the case that the ordered core satisfies a mild condition. It is worth mentioning that for the ordered core constructed in \cite[Example~5.4]{coredecreasing}, the following formula does not hold. Hence, some conditions on the core must be required.

\begin{theorem}
    Let $(U,\Sigma,\mu)$ be a measure space with a full $\sigma$-bounded ordered core $\mathcal{A}$ such that arbitrary unions and intersections in $\mathcal{A}$ are measurable in $\sigma(\mathcal{A})$. Then for any $\Sigma$-measurable function $g$ the formula
    $$
    \underline{g}(s) = \essinf_\mu \left\{ \left|g(v)\right|: v \leq_{\mathcal{A}} s \right\}
    $$
    holds.
\end{theorem}

\begin{proof}
    
 Let $h(s) = \essinf_\mu \left\{ \left|g(v)\right|: v \leq_{\mathcal{A}} s \right\}$. Since the order relation $\mathcal{A}$ is unchanged if we replace $\mathcal{A}$ by its maximal core, we may assume that $\mathcal{A}$ is maximal and that arbitrary unions and intersections of core sets are in the core, provided the result has finite $\mu$-measure. It follows from the definition of the order relation that
\begin{align*}
\left\{ t \in U: t <_{A} s \right\} &= \bigcup \left\{ A \in \mathcal{A}: s \in A \right\} \text{ and } \left\{ t \in U: t \leq_{A} s \right\} \\
&= \bigcap \left\{ A \in \mathcal{A}: s \not\in A \right\}.    
\end{align*}

By hypothesis, all of these sets are $\sigma(\mathcal{A})$-measurable for all $s \in U$. Define $[s] = \{ t \in U:  t \leq_{A} s \text{ and } s \leq_{A} t \}$, which is the difference of the sets above, so it is $\sigma(\mathcal{A})$-measurable as well.

To show that $h$ is a $\sigma(\mathcal{A})$-measurable function: Let $\alpha \in \mathbb{R}$ and define $O = h^{-1}\left(\alpha,\infty \right)$, we proceed to show that $O$ is $\sigma(\mathcal{A})$-measurable.  

Clearly $O \subseteq \bigcup_{x \in O} \left\{ t \in U: t \leq_{A} x \right\}$. Conversely, if $x \in O$ and $y \leq_{A} x$, then 
$$
h(y) = \essinf_\mu \left\{ t \in U: t \leq_{A} y \right\} \geq \essinf_\mu \left\{ t \in U: t \leq_{A} x \right\} = h(x) > \alpha,
$$
hence $\left\{ t \in U: t <_{A} x \right\} \subseteq O$, this proves that $O = \bigcup_{x \in O} \left\{ t \in U: t <_{A} x \right\}$, which by hypothesis, is a $\sigma(\mathcal{A})$-measurable set. As $\alpha$ was arbitrary, then $h$ is $\sigma(\mathcal{A})$-measurable. 

Since $h$ satisfies $y \leq_{A} x$ implies $h(y) \geq h(x)$ and is $\sigma(\mathcal{A})$-measurable, it only remains to show that $h$ is a minorant of $\left|g\right|$ and that it is optimal. 

We show the inequality $h(z) \leq \left|g(z)\right|$ by cases, depending on the measure of the set $[z]$. If $z \in U$ satisfies $\mu([z]) > 0$, notice that if $z' \in [z]$ then $h(z') = h(z)$. Hence, by definition of essential infimum we have that $\mu\left( \left\{z' \in [z]: \left|g(z')\right| < h(z) \right\} \right) = 0$. Therefore $h \leq \left|g\right|$ on $[z]$ up to a set of $\mu$-measure zero. Since $\lambda$ is a $\sigma$-finite measure, the collection of sets $U_D = \left\{ [z] : \mu([z]) > 0 \right\}$ must be countable. Hence, we have $h \leq \left|g\right|$ on its union up to a set of $\mu$-measure zero.

We must show the same inequality holds for the set $U_0 = \{z \in U: \mu([z]) = 0\}$. For this purpose: Fix $\epsilon > 0, n,m \in \mathbb{N}$, $\{A_n\} \in \mathcal{A}$ satisfy $U \subseteq \cup_n A_n$ and define 
$$S_{m,n} = \left\{ z \in U_0 \cap A_m: h(z) - \left|g(z)\right| > \epsilon \text{ and } n\epsilon \leq \left|g(z)\right| < (n+1)\epsilon \right\}.$$ 

By the previous estimate, we have that
$$
\mu \left( \{ z \in U: \left|g(z)\right| < h(z)\} \setminus \cup_{m,n} S_{m,n} \right) = 0.
$$
Since $U_D \in \sigma(\mathcal{A})$, is obtained by countably many unions of set differences of core sets, $R \chi_{U_D}$ is a characteristic function by \cite[Proposition~6.2(i)]{coredecreasing}. Since $U = U_0 \cup U_D$, we have that $R \chi_{U_0}$ is also a characteristic function, and $[0,\infty)$ is a disjoint union of some Borel sets $L_0,L_D$ such that $\chi_{L_0} = R \chi_{U_0}$ and $\chi_{L_D} = R \chi_{U_D}$. 

We claim that any $t \in [0,\infty)$ satisfying $\lambda(\{t\}) > 0$ must be contained in $L_D$. To see this, let $E_1,E_2$ satisfy $\mu(E_1) = \lambda(0,t)$ and $\mu(E_2) = \lambda(0,t]$, Observe that any $A \in \mathcal{A}$ must satisfy $\mu(A) \leq \mu(E_1)$ or $\mu(E_2) \leq \mu(A)$. Define
$$
M = \cup\{ A \in \mathcal{A}: \mu(A) < \mu(E_2)\} \quad \text{and} \quad N = \cap\{ A \in \mathcal{A}: \mu(E_1) < A\}.
$$
By hypothesis $M,N \in \mathcal{A}$, by the choice of $E_1,E_2$ we must have that $\mu(A) < \mu(E_2)$ implies $\mu(A) \leq \mu(E_1)$ and the monotone convergence theorem shows that $\mu(M) = \mu(E_1)$. Similarly, the dominated convergence theorem shows that $\mu(N) = \mu(E_2)$. Let $z \in M \setminus N$, then $\mu([z]) = \lambda({t}) > 0$, so $M \setminus N$ is contained in $U_D$. An application of $R$ yields $t \in L_D$.

Since the support of $R \chi_{S_{m,n}}$ is contained in $L_C$, there are no atoms, thus the function 
$$
\varphi(y) = \int\limits_{[y,\infty]} R \chi_{S_{m,n}} \, d\lambda
$$ 
is continuous. Moreover, $\varphi(0) = \mu(S_{m,n})$ and $\lim\limits_{y \to \infty} \varphi(y) = 0$.

Suppose that $\mu(S_{m,n}) > 0$ seeking a contradiction. Pick $r_1,r_2 > 0$ such that $\varphi(r_1) = \frac{\mu(S_{m,n})}{3}$, $\varphi(r_2) = \frac{\mu(S_{m,n})}{2}$ and let $E \in \mathcal{A}$ satisfy $r_1 \leq \mu(E) \leq r_2$. Then $\mu(S_{m,n} \cap E) > 0$ and $\mu(S_{m,n} \setminus E) > 0$. 
Let $z \in S_{m,n} \setminus E$, then any $t \in E$ satisfies $t \leq_{\mathcal{A}} z$, thus
\begin{align*}
h(z) &= \essinf_{\mu} \left\{ \left|g(t)\right|: t \leq_{\mathcal{A}} z \right\} \leq \essinf_{\mu} \left\{ \left|g(t)\right|: t \in E \right\} \\
&\leq \essinf_{\mu} \left\{ \left|g(t)\right|: t \in E \cap S_{m,n} \right\} \leq \epsilon(n+1).
\end{align*}
But since $z \in S_{m,n}$, we have $h(z) > \epsilon + \left|g(t)\right| > \epsilon + n\epsilon = (n+1)\epsilon$, which is a contradiction, therefore $\mu(S_{m,n}) = 0$ for all $m,n \in \mathbb{N}$. This shows that $h(z) \leq \left|g(z) \right|$ almost everywhere.

We have shown that $h$ is a core-decreasing minorant of $g$, thus $h \leq \underline{g}.$ To show the converse, let $z \in U$, and note that if $t \leq_{\mathcal{A}} z$, then $\underline{g}(z) \leq \underline{g}(t) \leq \left|g(t)\right|$, therefore taking essential infimum yields $\underline{g}(z) \leq h(z)$ completing the proof.   

\end{proof}

As a consequence of this result, we have the following examples where the ordered core satisfies that any arbitrary union or intersection of core sets can be reduced to a countable one, therefore it is measurable. These examples show that the terms appearing in formula (\ref{stepanov}) and \cite[Theorem~3.1]{ruzhansky22} are a particular case of the greatest core decreasing minorant.

\begin{example}
    Let $U = [0,\infty)$, $\mathcal{A} = \{ \emptyset \} \cup \left\{ [0,x] : x > 0\right\}$ and $\mu$ be a Borel measure, then 
    $$
    \underline{g}(x) = \essinf_{[0,x]} \left| g(t) \right|.
    $$
\end{example}

\begin{example}
    Let $U = \mathbb{X}$ be a metric measure space with distance function $d$, $a \in \mathbb{X}$ be any element, $\mu$ be any Borel measure and the core $$ \mathcal{A} = \{ \emptyset \} \cup \left\{ B_{a,r} : r > 0\right\}$$
    where $B_{a,r} = \{ x \in \mathbb{X}: d(a,x) \leq r \}$. Then $$\underline{g}(x) = \essinf_{\mu} \left\{ \left|g(t) \right|: t \in B_{a,\left|x\right|_a} \right\},$$
    where $\left|x\right|_a = d(a,x)$.
\end{example}

\section{Abstract Hardy inequalities with p=1}
Our approach to finding necessary and sufficient conditions on the measures for inequality (\ref{abstracthardy1}) is to find an equivalent inequality involving only two measures and a weight function, then to use Theorem \ref{pushmass} to replace the weight function with a core decreasing function. Finally, we find an equivalent Hardy inequality on the half line.

\begin{proposition}\label{primerareduccion}
    Fix $q \in (0,\infty)$, let $\eta$ and $\mu$ be $\sigma$-finite measures over $(U,\Sigma)$ and let $\tau$ be a $\sigma$-finite measure over $(Y,\tau)$. Suppose $B:Y\to\Sigma$ is a core map and $p=1$. Then there exists a positive $\Sigma$-measurable function $u$ such that the best constant in inequality ($\ref{abstracthardy1}$) is the same as the best constant in the inequality
    \begin{equation}\label{hardyred1}
            \Bigg( \int\limits_{Y} \bigg( \int\limits_{B(y)} f \, d\mu \bigg)^{q} \, d \tau(y)  \Bigg)^{1/q} \leq C \int\limits_{U} f u \, d\mu , \ \forall f \in L^{+}_\mu.  
    \end{equation}
\end{proposition}

\begin{proof}
First, we reduce the problem to the case $U = \cup_{y \in Y} B(y)$. Fix $f \in L^{+}(\Sigma)$, set $U_0 = \cup_{y \in Y} B(y)$ and $g = f \chi_{U_0}$. Then
\begin{align*}
    \frac{\Bigg( \int\limits_{Y} \bigg( \int\limits_{B(y)} g \, d\mu \bigg)^{q} \, d \tau(y)  \Bigg)^{1/q}}{\int\limits_{U_0} g \, d\eta} &= \frac{\Bigg( \int\limits_{Y} \bigg( \int\limits_{B(y)} f \, d\mu \bigg)^{q} \, d \tau(y)  \Bigg)^{1/q}}{\int\limits_{U} g \, d\eta} \\
    &\geq \frac{\Bigg( \int\limits_{Y} \bigg( \int\limits_{B(y)} f \, d\mu \bigg)^{q} \, d \tau(y)  \Bigg)^{1/q}}{\int\limits_{U} f \, d\eta}.
\end{align*}
Taking the supremum over all $f \in L^{+}(\Sigma)$ shows that
$$
\sup_{f \in L^{+}(\Sigma)} \frac{\Bigg( \int\limits_{Y} \bigg( \int\limits_{B(y)} f \, d\mu \bigg)^{q} \, d \tau(y)  \Bigg)^{1/q}}{\int\limits_{U} f \, d\eta} \leq \sup_{f \in L^{+}(\Sigma)} \frac{\Bigg( \int\limits_{Y} \bigg( \int\limits_{B(y)} f \, d\mu \bigg)^{q} \, d \tau(y)  \Bigg)^{1/q}}{\int\limits_{U_0} f \, d\eta}.  
$$
Conversely, 
\begin{align*}
\sup_{f \in L^{+}(\Sigma)} \frac{\Bigg( \int\limits_{Y} \bigg( \int\limits_{B(y)} f \, d\mu \bigg)^{q} \, d \tau(y)  \Bigg)^{1/q}}{\int\limits_{U_0} f \, d\eta} &= \sup_{f \chi_{U_0} \in L^{+}(\Sigma)} \frac{\Bigg( \int\limits_{Y} \bigg( \int\limits_{B(y)} f \, d\mu \bigg)^{q} \, d \tau(y)  \Bigg)^{1/q}}{\int\limits_{U} f \, d\eta} \\
&\leq \sup_{f \in L^{+}(\Sigma)} \frac{\Bigg( \int\limits_{Y} \bigg( \int\limits_{B(y)} f \, d\mu \bigg)^{q} \, d \tau(y)  \Bigg)^{1/q}}{\int\limits_{U} f \, d\eta}.
\end{align*}
Therefore, we may replace $U$ with $U_0$ in (\ref{Tabstracthardy1}). The same argument shows that we may replace $U$ with $U_0$ in (\ref{hardyred1}). Hence, we may suppose that $U = U_0$.

An application of the Lesbesgue decomposition theorem shows that $\mu = \mu_1 + \mu_2$, with $\mu_2 << \eta$ and $\mu_1 \perp \eta$. Also $U = U_1 \cup U_2$ with $U_1 \cap U_2 = \emptyset$ and $\mu_2(U_1) = 0 = \eta(U_2)$. The Radon-Nikodym theorem provides a $\Sigma$-measurable non-negative function $h$ such that $d\mu_2 = h \, d\eta$. If $E = \{ s \in U: h(s) = 0 \}$ we can define the function $g = h \chi_{(U \setminus E)}$ and the sets $V_1 = U_1 \setminus E$ and $V_2 = U_2 \cup E$ to get a decomposition $d\mu = g \, d\eta + d\mu_1$ supported on $V_1$ and $V_2$ respectively, moreover $g$ is never zero on $V_1$. Thus the inequality (\ref{abstracthardy1}) becomes
$$
\Bigg( \int\limits_{Y} \bigg( \int\limits_{B(y)} f g\, d \eta + \int\limits_{B(y)} f \, d \mu_1 \bigg)^q \, d\tau(y) \Bigg)^{\frac{1}{q}} \leq C \int_{U} f \, d\eta, \ \forall f \in L^{+}_\mu.
$$

Fix $z \in Y$ and set $f = \chi_{(B(z) \cap V_2)}$, then if $C$ is finite, we have
\begin{align*}
\Bigg( \int\limits_{Y} \bigg( \mu_1\left( B(y) \cap B(z) \right) \bigg)^q \, d\tau(y) \Bigg)^{\frac{1}{q}} &= \Bigg( \int\limits_{Y} \bigg( \int\limits_{B(y) \cap B(z)} \, d \mu_1 \bigg)^q \, d\tau(y) \Bigg)^{\frac{1}{q}} \\
&\leq C \eta\left(B(z) \cap V_2\right) = 0.    
\end{align*}
Therefore $\mu_1\left( B(y) \cap B(z) \right) = 0$ for $\tau$-almost every $y$. Since this holds for all $z \in Y$, letting $B(z) \uparrow U$ we get $\mu_1\left( U \right) = 0$. 

Hence the inequality becomes
$$
\Bigg( \int\limits_{Y} \bigg( \int\limits_{B(y)} f g\, d \eta  \bigg)^q \, d\tau(y) \Bigg)^{\frac{1}{q}} \leq C \int_{U} f \, d\eta, \ \forall f \in L^{+}_\mu.
$$

Since $g$ is non-zero $\eta$-almost everywhere, then we can define $u = \frac{1}{g}$, so $d \eta = u \, d\mu$. Notice that the sets $L^+_\mu$ and $L^+_\eta$ are only dependent on $\Sigma$, thus the substitution $f \mapsto fu$ is a bijection from $L^+_\eta \to L^+_\mu$ and yields the inequality (\ref{hardyred1}). This shows that if the best constant in the inequality (\ref{abstracthardy1}) is finite, then it is also the best constant in the inequality (\ref{hardyred1}). For the remaining case, notice that we can decompose $d\eta = u d\mu + d\eta_2$ for some measure $\eta_2$ satisfying $\eta \perp \eta_2$. Therefore
$$
\sup_{f \in L^{+}_\mu} \frac{\Bigg( \int\limits_{Y} \bigg( \int\limits_{B(y)} f \, d \mu  \bigg)^q \, d\tau(y) \Bigg)^{\frac{1}{q}}}{\int_{U} f \, d\eta} \leq \sup_{f \in L^{+}_\mu} \frac{\Bigg( \int\limits_{Y} \bigg( \int\limits_{B(y)} f \, d \mu  \bigg)^q \, d\tau(y) \Bigg)^{\frac{1}{q}}}{\int_{U} f u \, d\mu},
$$
thus if the best constant in inequality (\ref{hardyred1}) is infinite, then it is also the best constant in the inequality (\ref{abstracthardy1}) and completes the proof.

\end{proof}

Now we replace the weight function $u$ with its greatest core decreasing minorant.

\begin{proposition}\label{segundareduccion}
    Given a $\sigma$-finite measure $\mu$ over $(U,\Sigma)$, a $\sigma$-finite measure $\tau$ over $(Y,\tau)$, and a core map $B: Y \to \Sigma$, the best constant in inequality ($\ref{hardyred1}$) is the same as the best constant in the inequality
    \begin{equation}\label{hardyred2}
            \Bigg( \int\limits_{Y} \bigg( \int\limits_{B(y)} f \, d\mu \bigg)^{q} \, d \tau(y)  \Bigg)^{1/q} \leq C \int\limits_{U} f \underline{u} \, d\mu ,   
    \end{equation}
    where $\underline{u}$ is the greatest core-decreasing minorant of $u$ with respect to the ordered core $\mathcal{A} = \{\emptyset\} \cup \{ B(y): y \in Y\}. $
\end{proposition}

\begin{proof}
    
Our goal is to show that 
$$
\sup_{f \geq 0} \frac{\Bigg( \int\limits_{Y} \bigg( \int\limits_{B(y)} f \, d\mu \bigg)^{q} \, d \tau(y)  \Bigg)^{1/q}}{\int\limits_{U} f u \, d\mu} = \sup_{f \geq 0} \frac{\Bigg( \int\limits_{Y} \bigg( \int\limits_{B(y)} f \, d\mu \bigg)^{q} \, d \tau(y)  \Bigg)^{1/q}}{\int\limits_{U} f \underline{u} \, d\mu}.
$$

Since $\underline{u} \leq u$, the inequality `$\leq$' is clear. For the converse, using Theorem \ref{pushmass} we get
\begin{align*}
   \sup_{f \geq 0} \frac{\Bigg( \int\limits_{Y} \bigg( \int\limits_{B(y)} f \, d\mu \bigg)^{q} \, d \tau(y)  \Bigg)^{1/q}}{\int\limits_{U} f \underline{u} \, d\mu} &= \sup_{f \geq 0} \frac{\Bigg( \int\limits_{Y} \bigg( \int\limits_{B(y)} f \, d\mu \bigg)^{q} \, d \tau(y)  \Bigg)^{1/q}}{\inf \left\{ \int\limits_{U} g u \, d\mu : f \preccurlyeq g \right\}}\\
   &= \sup_{f \geq 0} \sup \left\{ \frac{\Bigg( \int\limits_{Y} \bigg( \int\limits_{B(y)} f \, d\mu \bigg)^{q} \, d \tau(y)  \Bigg)^{1/q}}{ \int\limits_{U} g u \, d\mu}  : f \preccurlyeq g \right\} \\
   &\leq \sup_{f \geq 0} \sup \left\{ \frac{\Bigg( \int\limits_{Y} \bigg( \int\limits_{B(y)} g \, d\mu \bigg)^{q} \, d \tau(y)  \Bigg)^{1/q}}{ \int\limits_{U} g u \, d\mu}  : f \preccurlyeq g \right\}. 
\end{align*}
Here the symbol $f \preccurlyeq g$ means that $\int_{E} f \, d\mu \leq \int_{E} g \, d\mu$ for every $E \in \mathcal{A}$.

The right hand side is bounded above by $\sup_{f \geq 0} \frac{\Bigg( \int\limits_{Y} \bigg( \int\limits_{B(y)} f \, d\mu \bigg)^{q} \, d \tau(y)  \Bigg)^{1/q}}{\int\limits_{U} f u \, d\mu}$, this completes the proof.

\end{proof}

We now reduce the problem to a Hardy inequality with measures over the half line.

\begin{lemma}\label{pasarreal}
    Given $B,\tau,\mu$ as in the previous propositions, then there exist Borel measures $\nu,\lambda$ on $[0,\infty)$ and a non-increasing function $w$ finite $\lambda$-almost everywhere, such that the best constant in inequality (\ref{hardyred2}) is the best constant in 
        \begin{equation}\label{hardyred3}
            \Bigg( \int\limits_{[0,\infty)} \bigg( \int\limits_{[0,x]} f \, d\lambda \bigg)^{q} \, d \nu(x)  \Bigg)^{1/q} \leq C \int\limits_{[0,\infty)} f w \, d\lambda , \ \forall f \in L^{+}_\lambda   
    \end{equation}
\end{lemma}

\begin{proof}
Since $B$ is a core map, then the function $\varphi: Y \to [0,\infty)$ defined by $\varphi(y) = \mu(B(y))$ is measurable. Let $\nu$ be the push-forward Borel measure associated to $\varphi$, that is
$$
\nu(E) = \tau\big( \varphi^{-1}(E)\big), \quad \forall E \text{ Borel.}
$$

Let $\lambda$ be the Borel measure associated to the ordered core $\mathcal{A}$ with enriched core $\mathcal{M}$, and $R,Q$ the transition operators.

Fix a positive $\Sigma$-measurable function $f$ integrable over every core set $A \in \mathcal{A}$ and define the functions 
$$
H f(x) = \int\limits_{[0,x]} R(f) \, d\lambda, \quad\mbox{and}\quad T f(y) = \int\limits_{B(y)} f \, d\lambda.
$$

We will show that $Hf$ and $T f$ are equimeasurable with respect to the measures $\nu$ and $\tau$ by computing their distribution functions. First notice that for all $y \in Y$ we have
$$
(Hf) \circ \varphi(y) = Hf\Big(\mu\big(B(y)\big)\Big) = \int\limits_{[0,\mu(B(y))]} R(f) \, d\mu = \int\limits_{B(y)} f \, d\lambda = Tf(y).
$$

Fix $\alpha > 0$ and define the sets
$$
E_{\alpha} = \left\{ x \in [0,\infty) : Hf(x) > \alpha \right\} \text{ and } F_{\alpha} = \left\{ y \in Y: T(y) > \alpha \right\}.
$$
Let 
$$
\gamma = \sup \left\{ x \in [0,\infty): \int\limits_{[0,x]} Rf\, d\lambda \leq \alpha \right\}.
$$
Notice that by the monotone convergence theorem $Hf(\gamma) \leq \alpha$. We claim that $E_\alpha = (\gamma,\infty)$ and that $F_\alpha = \varphi^{-1}(E_\alpha)$. 

Let $x \in E_\alpha$, then since $Hf$ is increasing, we must have that $x > \gamma$, thus $E_\alpha \subseteq (\gamma,\infty)$. Conversely, let $x > \gamma$, then $Hf(x) > \alpha$, thus $x \in E_\alpha$, this shows the first equation.

For the second equation, notice that 
$$
F_\alpha = \left\{ y \in Y: T(y) > \alpha \right\} = \left\{ y \in Y: (Hf) \circ \varphi(y) > \alpha \right\}.
$$
So if $y \in F_\alpha$, then $\varphi(y) \in E_\alpha$, this shows $F_\alpha \subseteq \varphi^{-1}(E_\alpha)$. Conversely, if $y \in \varphi^{-1}(E_\alpha)$, then $T(y) > \alpha$, hence $y \in F_\alpha$.

Computation of the distribution functions yields
$$
\nu(E_\alpha) = \tau\big(\varphi^{-1}(E_\alpha)\big) = \tau(F_\alpha).
$$
Therefore $Hf$ and $Tf$ are equimeasurable, hence
\begin{align*}
    \Bigg( \int\limits_{[0,\infty)} \bigg( \int\limits_{[0,x]} R(f) \, d\lambda \bigg)^q \, d\nu \Bigg)^{\frac{1}{q}} &= \Bigg( \int\limits_{[0,\infty)} \bigg( Hf \bigg)^q \, d\nu \Bigg)^{\frac{1}{q}} = \Bigg( \int\limits_{Y} \bigg( Tf \bigg)^q \, d\tau \Bigg)^{\frac{1}{q}} \\
    & = \Bigg( \int\limits_{Y} \bigg( \int\limits_{B(y)} f \, d \mu \bigg)^q \, d\tau \Bigg)^{\frac{1}{q}}. 
\end{align*}

Since $\underline{u}$ is core-decreasing, we have
$$
\int\limits_{U} f \underline{u} \, d\mu = \int\limits_{[0,\infty)} Rf R\underline{u} \, d\lambda.
$$
Therefore if inequality (\ref{hardyred2}) holds, so does
$$
\Bigg( \int\limits_{[0,\infty)} \bigg( \int\limits_{[0,x]} Rf \, d \lambda \bigg)^q \, d\nu(x) \Bigg)^{\frac{1}{q}} \leq C \int\limits_{[0,\infty)} Rf R\underline{u} \,  d\lambda, \quad \forall f \in L^{+}_\mu.
$$
Note that $R\underline{u}$ must be finite almost everywhere, otherwise, the original measures are not $\sigma$-finite. The result follows by letting $w = R\underline{u}$ and noting that $R$ maps $L^+_\mu$ onto $L^+_\lambda$.

\end{proof}

We are ready to prove the main result.

\begin{proof}\textit{(Of Theorem A)}
Suppose that $q \in (0,1)$, then by Lemma \ref{pasarreal} and Theorem \ref{stepanov} (Theorem 3.1 of \cite{stepanov08}) the best constant is equivalent to 
$$
\Bigg( \int\limits_{[0,\infty)} \bigg( \int\limits_{[0,x]} \frac{1}{\underline{w}} \, d\nu \bigg)^{\frac{q}{1-q}} \, d \nu(x)  \Bigg)^{1/q},
$$
where $w = R(\underline{u})$ and $\nu$ is the push-forward measure (see \cite{Boglibro}) for the map $\varphi(y) = \mu \circ B(y)$. Notice that $\underline{w} = w$, and it follows from Definition 2.2 (iv) that $\frac{1}{R(\underline{u})} = R\left(\frac{1}{\underline{u}}\right)$, then
\begin{align*}
\int\limits_{[0,x]} \frac{1}{\underline{w}} \, d\nu &= \int_{[0,\infty)} R\left(\frac{1}{\underline{u}}\right) \chi_{[0,x]} \, d\nu = \int_{Y} R\left(\frac{1}{\underline{u}}\right)\circ \varphi(y)  \chi_{[0,x]}\circ\varphi(y) \, d\tau(y) \\
&= \int\limits_{\varphi(y) \leq x} R\left(\frac{1}{\underline{u}}\right)\circ \varphi(y) \, d\tau(y).
\end{align*}
Thus
$$
\int\limits_{[0,\infty)} \bigg( \int\limits_{[0,x]} \frac{1}{\underline{w}} \, d\nu \bigg)^{\frac{q}{1-q}} \, d \nu(x)  = \int_{Y} \bigg( \int\limits_{\varphi(y) \leq \varphi(z)} R\left(\frac{1}{\underline{u}}\right)\circ \varphi(y) \, d\tau(y) \bigg)^{\frac{q}{1-q}} \, d\tau(z)
$$
and completes the proof for the case $q \in (0,1)$.

The case $q \in [1,\infty)$ follows directly from duality and we include it for the sake of completeness.

By Proposition \ref{primerareduccion} the best constant in inequality (\ref{Tabstracthardy1}) is the norm of the integral operator $K f(y) = \int_{U} k(y,s) f(s) \, d\theta(s)$ acting from $L^1_\theta \to L^{q}_\tau$ where $d\theta = \underline{u} d\mu$ and $k(y,s) = \frac{1}{\underline{u}(s)} \chi_{B(y)}(s)$. By duality, it is the best constant in the inequality
$$
\left\| \int\limits_{Y} k(y,\cdot) h(y) \, d\tau(y) \right\|_{L^{\infty}_\theta} \leq C  \left(\int_{Y} h^{q'} \, d\tau \right)^{\frac{1}{q'}}, \forall h \in L^{+}_\tau. 
$$
Define $\psi_s(y) = 1$ if $s \in B(y)$ and $\psi_s(y) = 0$ otherwise. Divide both sides of the equation by $\|h\|_{L^{q'}_\tau}$ to get
$$
\sup \left\{ \frac{1}{\underline{u}(s)} \int\limits_{Y} \psi_s(y) \frac{h (y)}{\|h\|_{L^{q'}_\tau}} \, d\tau(y) : s \in U \right\} \leq C.
$$
Taking supremum over non-zero positive functions $h$ yields
$$
\sup_{s \in U} \frac{1}{\underline{u}(s)} \| \psi_s \|_{L^{q}_\tau} \leq  C,
$$
which is the same as 
$$
C \geq \sup_{s \in U} \left( \frac{1}{\underline{u}}(s) \right) \tau\left( \left\{ y \in Y: s \in B(y)\right\} \right)^{1/q}.
$$
For the reverse inequality, an application of Minkowski's integral inequality yields
\begin{align*}
    \left( \int\limits_{Y} \left( \int_{U} k(s,y) f(s) \, d\theta(s) \right)^{q} \, d \tau(y)  \right)^{1/q} &\leq \left( \int\limits_{U} \left( \int_{Y} \psi_s(y) \, d\tau(y) \right)^{1/q} \frac{f(s)}{\underline{u}(s)}\, d\theta(s)  \right) \\
    &\leq \sup_{s \in U} \left( \frac{1}{\underline{u}}(s) \right) \tau\left( \left\{ y \in Y: s \in B(y)\right\} \right)^{1/q} \\
    &\times \int_{U} f(s) \, d\theta(s)
\end{align*}
hence $C \leq \sup_{s \in U} \left( \frac{1}{\underline{u}}(s) \right) \tau\left( \left\{ y \in Y: s \in B(y)\right\} \right)^{1/q}$ and proves the statement for $q \in [1,\infty)$.

\end{proof}

\section{Applications to metric measure spaces}
In this section, we show that the framework of abstract Hardy inequalities can be used to give different proofs to \cite[Theorem~2.1~Condition~$\mathcal{D}_1$]{ruzhansky19}, \cite[Theorem~2.1]{ruzhansky21} and \cite[Theorem~3.1]{ruzhansky22}. These theorems give necessary and sufficient conditions for Hardy inequalities to hold in metric measure spaces; they cover three cases depending on the indices $p$ and $q$, provided the existence of a locally integrable function $\lambda \in L^{1}_\text{loc}$ such that for all $f \in L^{1}(\mathbb{X})$ the following \textit{polar decomposition} at $a \in \mathbb{X}$ holds:
$$
\int_{\mathbb{X}} f d\mu = \int_{0}^\infty \int_{\Sigma_r} f(r,\omega) \lambda(r,\omega) \, d\omega_r dr,
$$
for a family of measures $d \omega_r$, where $\Sigma_r = \{ x \in \mathbb{X}: d(x,a) = r \}$.  

Our new proofs show that the polar decomposition hypothesis is not required so the results hold in all metric measure spaces.


\ 

We begin with the case $p > 1$, extending \cite[Theorem~2.1~Condition~$\mathcal{D}_1$]{ruzhansky19}, \cite[Theorem~2.1]{ruzhansky21} to all metric measure spaces. 

\begin{theorem}\label{reshardymetric}
    Let $\mu$ be a $\sigma$-finite measure on a metric measure space $\mathbb{X}$. Fix $a \in \mathbb{X}$ and let $p \in (1,\infty)$, $q>0$, $q\ne 1$ and $\omega,v$ be measurable functions, positive $\mu$-almost everywhere such that $\omega$ is integrable over $\mathbb{X} \setminus B_{a,\left|x\right|_a}$ and $v^{1-p'}$ is integrable over $B_{a,\left|x\right|_a}$ for all $x \in \mathbb{X}$. Then the Hardy inequality
\begin{equation}\label{pepon}
    \Bigg(\int\limits_{\mathbb{X}} \bigg( \int\limits_{B_{a,\left|x\right|_a}} f(y) \, d\mu(y) \bigg)^q \omega(x) \, d\mu(x) \Bigg)^\frac{1}{q} \leq C
    \bigg(\int\limits_{\mathbb{X}} f(x)^p v(x) \, d\mu(x) \bigg)^\frac{1}{p},
\end{equation}
holds for all $f \in L^{+}_\mu$ if and only if $p\leq q$ and
$$
\sup_{x \not= a} \left\{ \bigg(\int\limits_{\mathbb{X} \setminus B_{a,\left|x\right|_a}} \omega \, d\mu \bigg)^\frac{1}{q} \bigg(\int\limits_{B_{a,\left|x\right|_a}} v^{1-p'} \, d\mu \bigg)^\frac{1}{p'} \right\} < \infty,
$$
$0 < q < 1 < p$ and
$$
\int\limits_{\mathbb{X}} \bigg(\int\limits_{\mathbb{X} \setminus B_{a,\left|x\right|_a}} \omega \, d\mu \bigg)^\frac{r}{p} \bigg(\int\limits_{B_{a,\left|x\right|_a}} v^{1-p'} \, d\mu \bigg)^\frac{r}{p'} u(s) \, d\mu(s) < \infty,
$$
or $1 < q < p$ and
$$
\int\limits_{\mathbb{X}} \bigg(\int\limits_{\mathbb{X} \setminus B_{a,\left|x\right|_a}} \omega \, d\mu \bigg)^\frac{r}{q} \bigg(\int\limits_{B_{a,\left|x\right|_a}} v^{1-p'} \, d\mu \bigg)^\frac{r}{q'} v^{1-p'}(s) \, d\mu(s) < \infty.
$$
Here $\frac{1}{r} = \frac{1}{q} - \frac{1}{p}$.    
\end{theorem}

\begin{proof}

By hypothesis $v > 0$ and $v < \infty$ $\mu$-almost everywhere, then the mapping $f \mapsto v^{1-p'} f$ is a bijection on $L^{+}_\mu$. Then, the inequality (\ref{pepon}) is equivalent to 
\begin{equation*}
    \Bigg(\int\limits_{\mathbb{X}} \bigg( \int\limits_{B{a,\left|x\right|_a}} f(y) v^{1-p'}(y) \, d\mu(y) \bigg)^q \omega(x) \, d\mu(x) \Bigg)^\frac{1}{q} \leq C
    \bigg(\int\limits_{\mathbb{X}} f(x)^p v^{1-p'}(x) \, d\mu(x) \bigg)^\frac{1}{p},
\end{equation*}
Above we have used the identity $v^{p(1-p')} v = v^{1-p'}$. Let $d\tau = v^{1-p'} d \mu$ and define the map $B: \mathbb{X} \to \Sigma$ by 
$$
B(x) = B{a,\left|x\right|_a}.
$$
The image of $B$ is a totally ordered set. By hypothesis $\tau(B(x)) < \infty$ for each $x \in \mathbb{X}$. Therefore, it is an ordered core with respect to the measure $\tau$. We get the equivalent abstract Hardy inequality
$$
\Bigg( \int\limits_{\mathbb{X}} \bigg( \int\limits_{B_{a,\left|y\right|_a}} f \, d\tau \bigg)^{q} \omega(y) \, d\mu(y) \Bigg)^{\frac{1}{q}} \leq C \bigg( \int\limits_{\mathbb{X}}  f^p  \, d\tau  \bigg)^{\frac{1}{p}}, \ \forall f \in L^{+}_\mu.
$$
By definition of $\tau$ this is equivalent to
\begin{equation}\label{2measure}
    \Bigg( \int\limits_{\mathbb{X}} \bigg( \int\limits_{B_{a,\left|y\right|_a}} f v^{1-p'} \,  d\mu \bigg)^{q} \omega(y) \, d\mu(y) \Bigg)^{\frac{1}{q}} \leq C \bigg( \int\limits_{\mathbb{X}}  f^p  v^{1-p'} \, d\mu  \bigg)^{\frac{1}{p}}, \ \forall f \in L^{+}_\mu.
\end{equation}
Let $\lambda$ be the measure on $[0,\infty)$ induced by the core, so that for every $M$ in the core
$$
\int\limits_{[0,x]} R f \, d\lambda = \int\limits_{M} f v^{1-p'}\, d\mu, \text{ where } x = \int\limits_{M} v^{1-p'} \, d\mu. 
$$
We claim that inequality ($\ref{2measure}$) is equivalent to the Hardy inequality
\begin{equation}\label{2measurereal}
    \Bigg( \int\limits_{[0,\infty)} \bigg( \int\limits_{[0,y]} g \,  d\lambda \bigg)^{q} R \left(\frac{\omega}{v^{1-p'}}\right) \, d\lambda(y) \Bigg)^{\frac{1}{q}} \leq C \bigg( \int\limits_{[0,\infty)}  g^p   d\lambda  \bigg)^{\frac{1}{p}}, \quad \forall g \in L^{+}_\lambda.
\end{equation}

By \cite[Theorem~2.4]{sinnamon22}, it suffices to show that the normal form parameters of inequalities (\ref{2measure}) and (\ref{2measurereal}) coincide. Hence, it suffices to show that the maps
$$
b_1(s) = \int\limits_{B_{a,\left|s\right|_a}} v^{1-p'} \, d\mu \quad  \text{ and } \quad b_2(x) = \lambda([0,x]),
$$
have the same distribution functions with respect to the measures $\omega \, d\mu$ and $R \left(\frac{\omega}{v^{1-p'}}\right) \, d \lambda$ respectively. 

Fix $t > 0$ and consider the sets $E_1 = b_1^{-1}(t,\infty)$ and $E_2 = b_2^{-1}(t,\infty)$, we give a characterization for these sets.

Define the set $W$ as follows
$$
W = \bigcup \left\{ B_{a,\left|s\right|_a}: \int\limits_{B_{a,\left|s\right|_a}} v^{1-p'} \, d\mu \leq t \right\}.
$$
If $z \in E_1$, then $b_1(z) > t$, thus $z\not\in W$, conversely if $z \in W$ then $b_1(z) \leq t$, therefore $z \not\in E_1$. Hence $W^c = E_1$. Since $W$ is a union of closed balls centered at $a$, then there exists a sequence $s_n$ such that $B(a,s_n) \uparrow W$. Let $t_n = \int\limits_{B(a,s_n)} v^{1-p'} \, d\mu$.

Let $\widetilde{t}$ be defined as
$$
\widetilde{t} = \sup\left\{ z \leq t: z = \int\limits_{B_{a,\left|s\right|_a}} v^{1-p'} \, d\mu \text{ for some } s \in \mathbb{X} \right\},
$$
hence $\widetilde{t} = \lambda[0,t]$.

Therefore, the action of $R$ and two applications of the monotone convergence theorem show that
\begin{align*}
    \int\limits_{E_1^c} \omega \, d\mu &= \sup_{n \in \mathbb{N}} \int\limits_{B(a,s_n)} \omega \, d\mu = \sup_{n \in \mathbb{N}} \int\limits_{[0,t_n]} R\left(\frac{\omega}{v^{1-p'}}\right) \, d\lambda = \int\limits_{[0,t]} R\left(\frac{\omega}{v^{1-p'}}\right) \, d\lambda \\
    &= \int\limits_{E_2^c} R\left(\frac{\omega}{v^{1-p'}}\right) \, d\lambda.
\end{align*}

Since by hypothesis $\int\limits_{E_1^c} \omega \, d\mu < \infty$, then we have that
$$
\int\limits_{b_1^{-1}(t,\infty)} \omega \, d\mu = \int\limits_{b_2^{-1}(t,\infty)} R\left(\frac{\omega}{v^{1-p'}}\right) \, d\lambda.
$$
It follows that the distribution functions coincide which proves that the Hardy inequalities (\ref{2measure}) and (\ref{2measurereal}) have the same normal form parameter, therefore they are equivalent.

For all the index cases, we can apply \cite[Theorem~7.1]{sinnamon94} in the case $1 < p \leq q < \infty$, inequality (\ref{2measurereal}) holds if and only if
$$
\sup_x \bigg( \int\limits_{[x,\infty)} R\left( \frac{\omega}{v^{1-p'}} \right) \, d\lambda(t) \bigg)^{\frac{1}{q}} \bigg( \int\limits_{[0,x]} \, d\lambda  \bigg)^{\frac{1}{p'}} < \infty,
$$
which is equivalent to
$$
\sup_{s \not= a} \bigg( \int\limits_{\mathbb{X} \setminus B_{a,\left|s\right|_a}} \omega \, d\mu \bigg)^{\frac{1}{q}} \bigg( \int\limits_{B_{a,\left|s\right|_a}} v^{1-p'}\, d\mu  \bigg)^{\frac{1}{p'}} < \infty
$$

In the case $0 < q < 1 < p < \infty$, another application of \cite[Theorem~7.1]{sinnamon94} (be mindful of a typo in the exponents), inequality (\ref{2measurereal}) holds if and only if
$$
\int\limits_{[0,\infty)} \bigg( \int\limits_{[x,\infty)} R \left( \frac{\omega}{v^{1-p'}} \right) \, d\lambda \bigg)^{\frac{r}{p}} \bigg( \int\limits_{[0,x]} \, d\lambda \bigg)^{\frac{r}{p'}} R \left( \frac{\omega}{v^{1-p'}} \right) \, d\lambda(x) < \infty,
$$
which is equivalent to 
$$
\int\limits_{\mathbb{X}} \bigg( \int\limits_{\mathbb{X} \setminus B_{a,\left|s\right|_a}} \omega \, d\mu \bigg)^{\frac{r}{p}} \bigg( \int\limits_{B_{a,\left|s\right|_a}} v^{1-p'} \, d\mu \bigg)^{\frac{r}{p'}} \omega(s) \, d\mu(s) < \infty.
$$
In the case $1 < q < p$ we have that inequality (\ref{2measurereal}) holds if and only if
$$
\int\limits_{[0,\infty)} \bigg( \int\limits_{[x,\infty)} R \left( \frac{\omega}{v^{1-p'}} \right) \, d\lambda \bigg)^{\frac{r}{q}} \bigg( \int\limits_{[0,x]} \, d\lambda \bigg)^{\frac{r}{q'}} \, d\lambda(x) < \infty
$$
which is equivalent to
$$
\int\limits_{\mathbb{X}} \bigg( \int\limits_{\mathbb{X} \setminus B_{a,\left|s\right|_a}} \omega \, d\mu \bigg)^{\frac{r}{q}} \bigg( \int\limits_{B_{a,\left|s\right|_a}} v^{1-p'} \, d\mu \bigg)^{\frac{r}{q'}} v^{1-p'} \, d\mu(s) < \infty
$$
completing the proof.

\end{proof}

For the case $p = 1$, our Theorem A implies the following characterization

\begin{corollary}\label{reshardymetric2}
    Let $\mu$ be a $\sigma$-finite measure on a metric measure space $\mathbb{X}$. Fix $a \in \mathbb{X}$, let $q \in (0,\infty)$ and $\omega,v$ be measurable functions, positive $\mu$-almost everywhere satisfying  that $\omega$ is integrable over $\mathbb{X} \setminus B_{a,\left|x\right|_a}$ and  $v^{1-p'}$ is integrable over $B_{a,\left|x\right|_a}$ for each $x \in \mathbb{X}$. Then the best constant in the Hardy inequality
    $$
    \Bigg(\int\limits_{\mathbb{X}} \bigg( \int\limits_{B_{a,\left|x\right|_a}} f(y) \, d\mu(y) \bigg)^q \omega(x) \, d\mu(x) \Bigg)^\frac{1}{q} \leq C
    \int\limits_{\mathbb{X}} f(x) v(x) \, d\mu(x) , \ \forall f \in L^{+}_\mu
$$
satisfies
$$
C \approx \Bigg( \int\limits_{\mathbb{X}} \bigg( \int\limits_{z \leq_{\mathcal{A}} x} \frac{1}{\underline{v}}(x) \omega(x) \, d\mu(x)   \bigg)^{\frac{q}{1-q}} \omega(z) \, d\mu(z) \Bigg)^{\frac{1-q}{q}}, \text{ for } q \in (0,1),
$$
and
$$
C = \sup_{x \in X} \left( \frac{1}{\underline{v}}(x) \right) \bigg( \int\limits_{x \leq_{\mathcal{A}} t} \omega(t) \, d\mu(t) \bigg)^{1/q}, \text{ for } q \in [1,\infty).
$$
Here $\underline{v}(x) = \essinf_{\mu} \{ v(t): t \in B_{a,\left|x\right|_a} \}$, $x \leq_{\mathcal{A}} t$ means $B_{a,\left|x\right|_a} \subseteq B(a,\left|t\right|)$ and $B_{a,\left|x\right|_a} = \{ z \in \mathbb{X}: \text{dist}(a,z) \leq \text{dist}(a,x) \}$.
\end{corollary}

\begin{proof}
    
Let $\mathcal{A} = \{ \emptyset\} \cup \{ B_{a,\left|x\right|_a} \}_{x \in \mathbb{X}}$ be the full ordered core induced by the core map $x \to B_{a,\left|x\right|_a}$. Let $d\tau = \omega d\mu$, $d\eta = v d\mu$ and $\lambda$ be the measure on $[0,\infty)$ induced by the ordered core.

Consider the function $\varphi: \mathbb{X} \to [0,\infty)$ defined by $\varphi(x) = \mu\big(B_{a,\left|x\right|_a}\big)$ and let $\nu$ be the pushforward measure. Then, if $y = \varphi(x)$ we have
$$
\nu\big([0,y]\big) = \mu\big( \varphi^{-1}([0,y])\big) = \int\limits_{\varphi(t) \leq y} d\mu(t) = \int\limits_{B_{a,\left|x\right|_a}} d\mu = \lambda\big([0,\varphi(x)]\big) = \lambda\big([0,y]\big). 
$$
It follows that the Borel measures $\nu$ and $\lambda$ coincide and are finite over $[0,y]$ for all $y > 0$, therefore $\lambda$ is the pushforward measure of $\varphi$.

We now show that $R\Big(\frac{1}{\underline{v}}\Big) = \frac{1}{\underline{v}} \circ \varphi$ up to a set of $\mu$-measure zero.

Indeed
\begin{align*}
\int\limits_{B_{a,\left|x\right|_a}} \frac{1}{\underline{v}} \, d\mu &= \int\limits_{\varphi(t) \leq \varphi(x)} \frac{1}{\underline{v}}(t) \, d\mu(t) = \int\limits_{[0,\varphi(x)]} R\bigg(\frac{1}{\underline{v}}\bigg)(t) \, d\lambda(t) \\
&= \int\limits_{[0,\infty)} R\bigg(\frac{1}{\underline{v}}\bigg)(t) \chi_{[0,\varphi(x)]}(t) \, d\lambda(t)  = \int\limits_{\mathbb{X}} R\bigg(\frac{1}{\underline{v}}\bigg) \circ \varphi(t) \chi_{[0,\varphi(x)]} \circ \varphi(t) \, d\mu(t) \\
&= \int\limits_{\varphi(t) \leq \varphi(x)} R\bigg(\frac{1}{\underline{v}}\bigg) \circ \varphi(t) \, d\mu(t) = \int\limits_{B_{a,\left|x\right|_a}} R\bigg(\frac{1}{\underline{v}}\bigg) \circ \varphi(t) \, d\mu(t).  
\end{align*}

Since the equality holds for all core sets, then $R\Big(\frac{1}{\underline{v}}\Big) = \frac{1}{\underline{v}} \circ \varphi$ almost everywhere.

Then for $q \in (0,1)$, Theorem A yields
\begin{align*}
    C &\approx \Bigg( \int\limits_{\mathbb{X}} \bigg( \int\limits_{\varphi(z) \leq \varphi(x)} R \left( \frac{1}{\underline{v}} \right) \circ \varphi(x) \omega(x) \, d\mu   \bigg)^{\frac{q}{1-q}} \omega(z) \, d\mu(z) \Bigg)^{\frac{1-q}{q}} \\
    &\approx \Bigg( \int\limits_{\mathbb{X}} \bigg( \int\limits_{z \leq_{\mathcal{A}} x} \frac{1}{\underline{v}}(x) \omega(x) \, d\mu   \bigg)^{\frac{q}{1-q}} \omega(z) \, d\mu(z) \Bigg)^{\frac{1-q}{q}}.
\end{align*}

The statement for $q \in [1,\infty)$ follows directly from Theorem A. The description of $\underline{v}$ follows from Example 2.2 and completes the proof.  

\end{proof}

\section{Proof of Theorem \ref{pushmass}}

Before proving the functional description of the greatest core decreasing majorant, we need a technical lemma, which will be the key in the `pushing mass' technique needed to prove Theorem \ref{pushmass}. 

\begin{lemma}\label{minorant1}
    Let $u$ be a non-negative measurable function, $a > 0$ and $A = \{ s \in U: \underline{u}(s) \geq a \}$ such that $0 < \mu(A)$. 
    Then, for all $\delta > 0$ and $B \in \mathcal{M}$ such that $\mu(A) < \mu(B)$, the set
    $$
\{ s \in B \setminus A: \underline{u}(s) + \delta > u(s) \}
    $$
    has positive $\mu$-measure.
\end{lemma}

\begin{proof}
Since $\underline{u}$ is core-decreasing, up to a set of $\mu$-measure zero, if $s \in A$ and $t \leq_{\mathcal{A}} s$ then $t \in A$. Therefore $A$ coincides with a set in $\mathcal{M}$ up to measure zero. Suppose that the statement does not hold, then there exists some $\delta > 0$ and $B \in \mathcal{M}$ such that $\mu(A) < \mu(B)$ and $\underline{u}(s) + \delta \leq u(s)$ for $\mu$-almost all $s \in B \setminus A$.

Let $b = \essinf_{B \setminus A} \underline{u}(s)$, since $\underline{u}$ is core-decreasing, then $a > b$, equality does not hold, otherwise $\mu(B \setminus A) = 0$. Without loss of generality, we may assume that $\delta < a-b$, pick $n$ big enough, such that $\frac{a-b}{n} < \delta$ and define the function
$$
h = \underline{u} \chi_{\left( U \setminus (B \setminus A) \right)} + \sum_{k=1}^{n} \left( b + k\frac{a-b}{n}\right) \chi_{\left(E_{k-1} \setminus E_{k}\right)},
$$
where $E_k = \{ s \in U: \underline{u}(s) \geq b + k\frac{a-b}{n} \}$. Notice that $h$ is core-decreasing by construction and $h \geq \underline{u}$ but $h(s) - \underline{u}(s) < \delta$, hence $h$ is also a minorant of $u$, by maximality we get $h = \underline{u}$. Since $\mu(B \setminus A) > 0$, there exists some $k$ such that $\mu(E_{k-1} \setminus E_{k}) > 0$, and notice that $k \not= n$, now define
$$
h_2 = \underline{u} \chi_{\left( U \setminus (E_{k-1} \setminus E_{k} \right)} + \left( b + (k+1)\frac{a-b}{n}\right) \chi_{\left(E_{k-1} \setminus E_{k}\right)}.
$$
By the same argument as before, $h_2$ is a core-decreasing minorant of $u$, but $h_2$ is strictly greater than $\underline{u}$, a contradiction.

\end{proof}

We now `push the mass to the left' of $f$ to an appropriate function $g$ to achieve the desired infimum.

\begin{lemma}\label{pushmasslem2}
Let $u$ and $f$ be non-negative measurable functions such that the integral $\int_{U} f \underline{u} \, d\mu$ is finite. Then, for any $\epsilon > 0$, there exists a measurable non-negative function $g$ such that $\int_E g \, d\mu \geq \int_E f \, d\mu $ for any $E \in \mathcal{A}$ and 
    $$
\int_U g u \, d\mu - \epsilon < \int_U f \underline{u} \, d\mu.
    $$
\end{lemma}

\begin{proof}
Fix $\epsilon > 0$. Since we assume that $\int_{U} f \underline{u} \, d\mu < \infty$, there exists $\alpha > 1$ such that 
$$
\alpha \int_{U} f \underline{u} \, d\mu < \int_{U} f \underline{u} \, d\mu + \frac{\epsilon}{2}
$$
Define the sequence $\{A_n\}_{n \in \mathbb{Z}}$ as
$$
A_{n} = \left\{ s \in U : \underline{u}(s) \geq \alpha^{n+1} \right\}, \quad\text{ for each } n \in \mathbb{Z}.
$$
Since $\underline{u}$ is core decreasing, the sets $A_{n} \in \mathcal{M}$. Define the sets $\{J_{n}\}_{n \in \mathbb{Z} \cup \{\pm \infty\}}$ by
$$
J_{\infty} = \bigcap_{n \in \mathbb{Z}} A_{n}, \quad J_n = A_{n} \setminus A_{n+1}, \quad\text{for each } n \in \mathbb{Z}, \quad \text{and}\quad J_{-\infty} = U \setminus \bigcup_{n \in \mathbb{Z}} J_{n}.    
$$
Notice that the sets $\{J_{n}\}_{n \in \mathbb{Z} \cup \{\pm \infty\}}$ are disjoint and cover the whole space $U$. Also $J_{\infty} \in \mathcal{A}$ and the complement of $J_{-\infty}$ also belongs in $\mathcal{A}$. It will be useful to consider the presentation
$$
J_{\infty} = \left\{ s \in U: \underline{u}(s) = \infty \right\}, \quad J_{-\infty} = \left\{ s \in U: \underline{u}(s) = 0 \right\}, \quad \text{and}
$$
$$
J_n = \left\{ s \in U: \alpha^n \leq \underline{u}(s) < \alpha^{n+1} \right\}, \quad\text{for each } n \in \mathbb{Z}.    
$$
Define the functions $f_n = f \chi_{J_n}$ for each $n \in \mathbb{Z} \cup \{\pm \infty\}$. Our goal is to build non-negative functions $g_{n}$ satisfying
\begin{align}
    \int_{E} g_n \, d\mu &\geq \int_{E} f_{n} \, d\mu, \quad \text{for all } E \in \mathcal{A} \quad \text{and each } n \in \mathbb{Z} \cup \{\pm \infty\}, \label{pushmassdom}\\
    \int_{U} g_n u \, d\mu &\leq \alpha \int_{U} f_{n} \underline{u} \, d\mu, \quad \text{for each } n \in \mathbb{Z} \cup \{\infty\}, \text{ and }\label{pushmassalpha}\\
    \int_{U} g_{-\infty} u \, d\mu &\leq \frac{\epsilon}{2}.  \label{pushmasssmall}
\end{align}
Since $f = \sum_{n \in \mathbb{Z} \cup \left\{\pm \infty \right\}} f_n$, the function $g = \sum_{n \in \mathbb{Z} \cup \left\{\pm \infty \right\}} g_n$ clearly satisfies $\int_{E} g \, d\mu \geq \int_{E} f \, d\mu$ for all $E \in \mathcal{A}$ and
\begin{align*}
\int_{U} g u \, d\mu  &= \sum_{n \in \mathbb{Z} \cup \left\{\pm \infty \right\}} \int_{U} g_n u \, d\mu \leq \alpha \sum_{n \in \mathbb{Z} \cup \left\{\infty \right\}} \int_{U} f_n \underline{u} \, d\mu + \frac{\epsilon}{2} \\
&= \alpha \int_{U} f \underline{u} \, d\mu + \frac{\epsilon}{2} < \int_{U} f \underline{u} \, d\mu + \epsilon.
\end{align*}
For any $n \in \mathbb{Z} \cup \{\infty\}$  such that $f_n = 0$ we define $g_n = 0$ and it clearly satisfies inequalities (\ref{pushmassdom}) and (\ref{pushmassalpha}). For the other cases, since $\int_{U} f \underline{u} \, d\mu < \infty$, we must have that $f_{\infty} = 0$.

Fix $n \in \mathbb{Z}$ such that that $f_n \not=0$ $\mu$-almost everywhere. This means that $0 < \mu(J_n) = \mu(A_n) - \mu(A_{n+1})$. Since 
$$
\infty > \int_{U} f \underline{u} \, d\mu \geq \int_{J_{n}} f \underline{u} \, d\mu \geq \alpha^n \int_{J_{n}} f \, d\mu,
$$
thus $\int_{J_n} f\, d\mu < \infty$.

Let 
$$
\beta_n = \inf\left\{ \mu(E): \mu(A_{n+1}) < \mu(E), E \in \mathcal{M}, \text{ and } E \subseteq A_n  \right\}.
$$
There are two cases, either $\beta_n > \mu(A_{n+1})$ or $\beta_n = \mu(A_{n+1})$. In the first case, pick $C_n \in \mathcal{M}$ such that $\mu(C_n) = \beta_n$. An application of Lemma \ref{minorant1} with $a = \alpha^{n+1}$, $b = \alpha^n$ $B = A_n$, $A = A_{n+1}$ and $C = C_{n}$ shows that the set
$$
H_{n} = \left\{ s \in (C_n \setminus A_{n+1}): \alpha^{n} \leq \underline{u}(s) \leq u(s) < \alpha^{n+1} \right\}
$$
has positive $\mu$-measure. Define $$g_n = \Big( \int_{J_n} f \, d\mu \Big) \frac{\chi_{H_n}}{\mu(H_n)}.$$ For each $E \in \mathcal{A}$,
$$
\int_{E} g_n \, d\mu = \begin{cases}
			0, & \text{if $\mu(E) \leq \mu(A_{n+1})$}\\
            \int_{J_n} f \, d\mu, & \text{otherwise}
		 \end{cases} = \int_{E} f_n \, d\mu,
$$
therefore $g_n$ satisfies the inequality (\ref{pushmassdom}). Also 
$$
\int_{U} g_n u \, d\mu = \frac{ \int_{J_n} f \, d\mu}{\mu(H_n)} \int_{H_n} u \, d\mu < \alpha^{n+1} \int_{J_n} f \, d\mu = \alpha \int_{J_n} \alpha^{n} f \, d\mu \leq \alpha \int_{J_n} f_n \underline{u}  \, d\mu,
$$
proving that $g_n$ satisfies the inequality (\ref{pushmassalpha}).

The remaining case is when $\beta_n = \mu(A_{n+1})$. We prove by induction that there exists a sequence of sets $\{H_{n,m}\}_{m \in \mathbb{N}^+}$ of positive $\mu$-measure and $\{C_{n,m}\}_{m \in \mathbb{N}}$ such that $C_{n,m} \subseteq A_{n}$, $\mu(C_{n,m})$ is strictly decreasing to $\mu(A_{n+1})$ and
$$
H_{n,m} \subseteq \left\{ s \in C_{n,m-1} \setminus C_{n,m}: \alpha^{n} \leq \underline{u}(s) \leq u(s) < \alpha^{n+1}\right\}.
$$
We show the induction step first. Suppose that the sequences are constructed up to an integer $M_0 > 0$. Apply Lemma \ref{minorant1} with $a = \alpha^{n+1}$, $b = \alpha^n$ $B = A_n$, $A = A_{n+1}$ and $C = C_{n,M_0}$, to get that the set
$$
K_{M_0} = \left\{ s \in (C_{n,M_0} \setminus A_{n+1}): \alpha^{n} \leq \underline{u}(s) \leq u(s) < \alpha^{n+1} \right\}
$$
has positive $\mu$-measure. Since $\beta_n = \mu(A_{n+1})$, there exists a set $C_{n,M_0+1} \in \mathcal{M}$ such that $\mu(C_{n,M_0+1} \setminus A_{n+1}) < \frac{\mu(K_{M_0})}{2}$. Another application of Lemma \ref{minorant1} with $a = \alpha^{n+1}$, $b = \alpha^n$ $B = A_n$, $A = A_{n+1}$ and $C = C_{n,M_0+1}$ provides a set
$$
K_{M_0 + 1} = \left\{ s \in (C_{n,M_0+1} \setminus A_{n+1}): \alpha^{n} \leq \underline{u}(s) \leq u(s) < \alpha^{n+1} \right\}
$$
of positive $\mu$-measure. Notice that $K_{M_0 + 1} \subseteq K_{M_0}$ but $\mu(K_{M_0 + 1}) < \mu(K_{M_0})$, therefore the difference has positive measure. Set $H_{n,M_0+1} = K_{M_0+1} \setminus K_{M_0}$ to prove the induction step. The base case follows the same argument, letting $M_0 = 0$ and $C_{n,0} = A_n$. 

Define the function
$$
g_n = \sum_{m=2}^{\infty} \Big( \int\limits_{C_{n,m-2} \setminus C_{n,m-1}}  f\, d\mu \Big) \frac{\chi_{H_{n,m}}}{\mu(H_{n,m})}.
$$
Let $E \in \mathcal{A}$, if $\mu(E) \leq \beta_n$, then both $\int_{E} f_n \, d\mu$ and $\int_{E} g_n \, d\mu$ vanish. If $\mu(E) \geq \mu(A_{n})$, then 
$$
\int_{E} g_n \, d\mu = \sum_{m=2}^{\infty} \Big( \int\limits_{C_{n,m-2} \setminus C_{n,m-1}}  f\, d\mu \Big) = \int_{J_n} f\, d\mu = \int_{E} f_n \, d\mu.
$$
In the case that $\mu(E) \in \big(\beta_n,\mu(A_n)\big)$, there exists some $M_E \in \mathbb{N}$ such that $\mu(E) \in \big(\mu(C_{M_{E}+1}),\mu(C_{M_{E}})\big]$, hence
\begin{align*}
\int_{E} g_n \, d\mu &\geq \int\limits_{C_{M_E + 1}} g_n \, d\mu  = \sum_{m=C_{M_E+1}}^{\infty} \Big( \int\limits_{C_{n,m-2} \setminus C_{n,m-1}}  f\, d\mu \Big) = \int\limits_{C_{M_E} \setminus A_{n+1}} f\, d\mu \geq \int\limits_{E \setminus A_{n+1}} f\, d\mu \\
&= \int_{E} f_n \, d\mu.
\end{align*}
Therefore $g_n$ satisfies the inequality (\ref{pushmassdom}). Also 
\begin{align*}
\int_{U} g_n u \, d\mu &= \sum_{m=2}^{\infty} \Big( \int\limits_{C_{n,m-2} \setminus C_{n,m-1}}  f\, d\mu \Big) \frac{\int_{H_m} u \, d\mu}{\mu(H_{m})} \leq  \sum_{m=2}^{\infty} \Big( \int\limits_{C_{n,m-2} \setminus C_{n,m-1}}  f\, d\mu \Big) \alpha^{n+1} \\
&= \alpha_{n+1} \int_{J_n} f \, d\mu = \alpha \int_{J_n} \alpha^n f \, d\mu \leq \alpha \int_{J_n} f_n \underline{u}.
\end{align*}
proving that $g_n$ satisfies the inequality (\ref{pushmassalpha}).

All that remains is defining the function $g_{-\infty}$ whenever the function $f \chi_{J_{-\infty}}$ is not zero $\mu$-almost everywhere. Let $U_0 = \cup_{n \in \mathbb{Z}} J_n $. Since $\mu(J_{-\infty}) > 0$, then there exists some $E \in \mathcal{A}$ such that $U_0 \subseteq A$, therefore $\mu(U_0) < \infty$, thus $U_0 \in \mathcal{M}$. If there exists a set of positive measure $W$ such that $u(s) = 0$ for all $s \in W$ and $\mu(W \cap E) > 0$ for every $E$ satisfying $\mu(E) > \mu(U_0)$ then we define 
$$
g_{-\infty} = \infty \chi_{W}. 
$$
In this case $\int_{U} g_{-\infty}u \, d\mu = 0$, clearly satisfying the inequality (\ref{pushmasssmall}). For any $E \in \mathcal{A}$, if $\mu(E) \leq \mu(U_0)$ then $$0 = \int_{E} f_{-\infty} \leq \int_{E} g_{-\infty}$$ and if $\mu(E) > \mu(U_0)$, then $$\infty = \int_{E} g_{-\infty} \, d\mu \geq \int_{E} f_{-\infty} \, d\mu,$$ thus $g_{-\infty}$ satisfies the inequality (\ref{pushmassdom}).

If such set $W$ does not exist, we will find a disjoint sequence of sets of positive measure $\{W_k\}_{k \in \mathbb{N}^+}$, such that
\begin{equation}\label{dobleuk}
W_k \subseteq \left\{ s \in U: u(s) < \epsilon 2^{-(k+1)} \right\},
\end{equation}
also satisfying that for any $E \in \mathcal{A}$ such that $\mu(E) > \mu(U_0)$, then infinitely many sets in the sequence are subsets of $E$. The desired function will be
$$
g_{-\infty} = \sum_{k=1}^{\infty} \frac{\chi_{W_k}}{\mu(W_k)}.
$$
Then 
$$
\int_{U} g_{-\infty}u \, d\mu = \sum_{k=1}^{\infty} \frac{\int_{W_k} u \, d\mu}{\mu(W_k)} < \frac{\epsilon}{2} \sum_{k=1}^{\infty} 2^{-k} = \frac{\epsilon}{2},
$$ satisfying the inequality (\ref{pushmasssmall}). For any $E \in \mathcal{A}$, if $\mu(E) \leq \mu(U_0)$ then $$0 = \int_{E} f_{-\infty} \leq \int_{E} g_{-\infty}$$ and if $\mu(E) > \mu(U_0)$, then $$\infty = \sum_{W_{k} \subseteq E} 1 = \int_{E} g_{-\infty} \, d\mu \geq \int_{E} f_{-\infty} \, d\mu,$$ thus $g_{-\infty}$ satisfies the inequality (\ref{pushmassdom}). We now show that either the set $W$ exists or we build the sequence $\{W_k\}$.

Since $A_{-n}$ increases to $U_0$ whenever $n \uparrow \infty$, we have that $\mu(A_{-n}) \uparrow \mu(U_0)$. There are two possibilities, either $\mu(A_{-n}) < \mu(U_0)$ for all $n \in \mathbb{N}$ or there exists some $N_0$ such that $\mu(A_{-N_0}) = \mu(U_0)$. 

In the first case, for any $j \in \mathbb{N}^+$ the set
$$
G_{j} = \left\{ s \in U_0: u(s) < \epsilon 2^{-(j+1)} \right\},
$$
has positive $\mu$-measure. Otherwise, the function $h = \epsilon 2^{-(j+1)} \chi_{U_0}$ is a core decreasing minorant of $u$, thus $h \leq \underline{u}$. Hence, for any $n$ large enough such that $\alpha^{-n} < \epsilon 2^{-(j+1)}$ we would have $\mu(A_{-n}) = \mu(U_0)$ arriving at a contradiction. Notice that $\{G_j\}$ is a decreasing sequence, let $W = \cap G_j$. If $\mu(W) > 0$, then there is nothing left to prove. If $\mu(W) = 0$, then we may choose a subsequence $\{G_{j_k}\}_k$ such that the sequence of measures $\{\mu(G_{j_k})\}$ is strictly decreasing. Then the sequence $W_k = G_{j_k} \setminus G_{j_{k+1}}$ is disjoint and satisfies formula (\ref{dobleuk}). 

It remains to show that the set $W$ or the sequence $\{W_k\}$ exist whenever there is a positive integer $k_0$ such that $\mu(A_{-k_0}) = \mu(U_0)$, which implies that $\alpha^{-k_0} \leq \underline{u}(s) \leq u(s)$ for almost every $s \in U_0$. Let $$\beta_{-\infty} = \inf\left\{ \mu(E): E \in \mathcal{M} \text{ and }\mu(E) > \mu(U_0) \right\}.$$ Once more, we consider the two possibilities; $\mu(U_0) < \beta_{-\infty}$ or if $\mu(U_0) = \beta_{-\infty}$.

In the first case, let $A_{-\infty} \in \mathcal{M}$ satisfy $\mu(A_{-\infty}) = \beta_{-\infty}$. Pick $r_0$ large enough, such that $\epsilon 2^{-r_0} < \alpha^{-k_0}$. For any $j > r_0$, apply Lemma \ref{minorant1} with $a = \epsilon 2^{-j}$, $b = 0$, $A = U_0$, $B = U$ and $C = A_{-\infty}$ to get that the set  
$$
G_{j} = \left\{ s \in (A_{-\infty} \setminus U_0): u(s) < \epsilon 2^{-j} \right\} 
$$
has positive $\mu$-measure. Let $W = \cap_{k > r_0} G_j$. If $\mu(W) > 0$ there is nothing to prove, so we drop to a subsequence with strictly decreasing measures and build the sequence $\{W_k\}$ like it was done before. Note that for any $E \in \mathcal{A}$ satisfying $\mu(E) > \mu(U_0)$, then $\mu(A_{-\infty}) \leq \mu(E)$, so every set in the sequence $W_k$ is contained in $E$.

We are left with the final case; when $\mu(U_0) = \beta_{-\infty}$. Choose a sequence $\{ E_j \} \subseteq \mathcal{M}$ such that $\mu(E_j) \downarrow \beta_n$, and $r_0$ large enough so $\epsilon 2^{-r_0} < \alpha^{-k_0}$. For any $j > r_0$, apply Lemma \ref{minorant1} with $a = \epsilon 2^{-j}$, $b = 0$, $A = U_0$, $B = U$ and $C = E_{j}$ to get that the set  
$$
G_{j} = \left\{ s \in (E_j \setminus U_0): u(s) < \epsilon 2^{-j} \right\} 
$$
has positive $\mu$-measure. Since $\mu(G_j) \leq \mu(E_j \setminus U_0)$, we get $\mu(G_j) \downarrow 0$. Once more we can drop to a subsequence and repeat the previous process to obtain disjoint sets of positive measure $W_j$ satisfying formula (\ref{dobleuk}) such that $W_j \subseteq E_j$. Therefore, for any $E$ such $\mu(E) > \mu(U_0)$, there are infinitely many of such sets $W_j$ contained in $E$. This finishes the proof.

\end{proof}

With this, we finish the functional description of the greatest core decreasing minorant

\begin{proof}\textit{(Of Theorem \ref{pushmass})}
    
If $g$ satisfies $\int_E g \, d\mu \geq \int_E f \, d\mu \text{ for all } E \in \mathcal{A}$ then
\begin{align*}
    \int_U g u \, d\mu &\geq \int_U g \underline{u} \, d\mu \text{ since } u \geq \underline{u}\\
    &\geq \int_U f \underline{u} \, d\mu \text{ since } \underline{u} \text{ is core-decreasing}.
\end{align*}

Infimum over all $g$ yields the inequality $$\int_U f \underline{u} \, d\mu \leq \inf\left\{ \int_U g u \, d\mu : \int_E g \, d\mu \geq \int_E f \, d\mu \text{ for all } E \in \mathcal{A}\right\}.$$ If $\infty = \int_{U} f \underline{u} \, d\mu$, then equality clearly follows. So we may suppose that $\int_{U} f \underline{u} \, d\mu < \infty$, and in this case, equality follows from Lemma \ref{pushmasslem2}.

\end{proof}

\vskip2pc


\end{document}